\documentclass[11pt,twoside,a4paper]{article}
\usepackage[english]{babel}
\usepackage{amsmath,amssymb,amsthm,epsfig,color,graphicx}

\usepackage{subfigure}

\usepackage{pstricks}
\usepackage{pst-plot}
\usepackage{pst-all}
\usepackage{diagbox}
 \usepackage[title, titletoc]{appendix} 
\usepackage{float}

\usepackage{color}

\newtheorem{theorem}{Theorem}
\newtheorem{proposition}{Proposition}
\newtheorem{lemma}{Lemma}
\newtheorem{definition}{Definition}

\newtheorem{remark}{Remark}

\definecolor{Red}{cmyk}{0,1,1,0}

\definecolor{Blue}{cmyk}{1,1,0,0}

\begin{document}

\title{Hyperbolicity and genuine nonlinearity 
conditions for certain p-systems of conservation laws, weak solutions and the entropy condition}
\author{Edgardo P\'erez\\
\footnotesize{D\'epartement de Math\'ematiques}\\
\footnotesize{Universit\'e de Brest, 6, rue Victor le Gorgeu, 29285 Brest, France}\\
\footnotesize{\texttt{edgardomath@gmail.com}}\\
\\[0.3cm]
Krzysztof R\'ozga\\
\footnotesize{Department of Mathematical Sciences}\\
\footnotesize{University of Puerto Rico at Mayaguez, Mayaguez,Puerto Rico 00681-9018}\\
\footnotesize{\texttt{krzysztof.rozga@upr.edu}}
}
\maketitle


\begin{abstract}

We consider a p-system of conservation laws that emerges in one dimensional elasticity theory. 
Such system is determined by a function $W$, called strain-energy function. 
We consider four forms of $W$ which are known in the literature. 
These are St.Venant-Kirchhoff, Ogden, Kirchhoff modified, Blatz-Ko-Ogden forms. 
In each of those cases we determine the conditions for the parameters $\rho_0$, $\mu $ and
 $\lambda$, under which the corresponding system is hyperbolic and genuinely nonlinear.
We also establish what it means a weak solution of an initial and boundary value problem. Next we concentrate on a particular problem whose weak solution is obtained in a linear theory by means of D'Alembert's formula. In cases under consideration the p-systems are nonlinear, so we solve them employing Rankine-Hugoniot conditions. Finally we ask if such solutions satisfy the entropy condition. For a standard entropy function we provide a complete answer, except of the Blatz-Ko-Ogden case. For a general strictly convex entropy function the result is that for the initial value of velocity function near zero these solutions satisfy the entropy condition, under the assumption of hyperbolicity and genuine nonlinearity.

\end{abstract}


\section{Introduction}

The mathematical theory of hyperbolic systems of conservation laws were
started by Eberhardt Hopf in $1950$, followed in a series of studies by Olga
Oleinik, Peter D. Lax and James Glimm \cite{LAX2} . The class of
conservation laws is a very important class of partial differential
equations because as their name indicates, they include those equations that
model conservation laws of physics (mass, momentum, energy, etc).



As important examples of hyperbolic systems of balance laws arising in
continuum physics we have: Euler's equations for compressible gas flow, the
one dimensional shallow water equations \cite{EVA}, Maxwell's equations in
nonlinear dielectrics, Lundquist's equations of magnetohydrodynamics and
Boltzmann equation in thermodynamics \cite{CONS} and equations of elasticity
\cite{MAR}.

One of the main motivations of the theory of hyperbolic systems is that they
describe for the most part real physical problems, because they are
consistent with the fact that the physical signals have a finite propagation
speed \cite{MAR}. Such systems even with smooth initial conditions may fail
to have a solution for all time, in such cases we have to extend the concept
of classical solutions to the concept of a weak solution or generalized
solution \cite{EVA}.

In the case of hyperbolic systems, the notion of weak solution based on
distributions does not guarantee uniqueness, and it is necessary to devise
admissibility criteria that will hopefully single out a unique weak
solution. Several such criteria have indeed been proposed, motivated by
physical and/or mathematical considerations. It seems that a consensus has
been reached on this issue for such solutions, they are called entropy
conditions \cite{DAFER2}. Nevertheless, to the question about existence and
uniqueness of generalized solutions subject to the entropy conditions, the
answer is, in general, open. For the scalar conservation law, the questions
existence and uniqueness are basically settled \cite{EVA}. For genuinely
nonlinear systems, existence (but not uniqueness) is known for initial data
of small total variation \cite{RINAR}. Some of the main contributors to the
field are Lax , Glimm , DiPerna, Tartar, Godunov, Liu, Smoller and Oleinik
\cite{LAX1}, \cite{GLIM}, \cite{OLEI} .

All of this motivates us to study systems of conservation laws that emerge
in the theory of elasticity. These systems are determined by constitutive
relations between the stress and strain. For hyperelastic materials, the
constitutive relations can be written in a simpler form. Now the stress is
determined by a scalar function of the strain called the strain-energy
function $W$. A further simplification of a stress-strain relation is
obtained for isotropic materials.

In applications some specific strain-energy functions are used; in our work
we consider four different forms of $W$. In all our studies we restrict
ourselves to the case of one dimensional elasticity. \

The first important question that arises is the following: \ given the
function $W$, is the corresponding system of PDE's hyperbolic? By answering
it, we can assess how good the model corresponding to that particular $W$ is.

There exists also another important condition called genuine nonlinearity
condition, which is related to the entropy condition, \cite{RINAR}.
According to our previous remarks the entropy condition can be considered a
physical one. This implies an importance of genuine nonlinearity condition
as well. For that reason our second question is about the validity of that
particular condition for the models under study.

Our third important question is how manageable is the entropy condition, that
is, given a weak solution of the elasticity system, can we
conclude if it is or not an entropy solution? In general, except of the
linear case, it is not easy to answer that question, because in the entropy
condition there appear two functions: entropy and entropy-flux, which
satisfy a given nonlinear system of PDE's, the first of them is convex and
otherwise they are arbitrary.

For this reason we restrict ourselves to study the entropy condition for a
relatively simple weak solutions, which correspond to a well understood
physical situation of what can be called a compression shock. Such solutions
are obtained easily in linear case by means of D'Alembert's formula and by
analogy in nonlinear case, employing the Rankine-Hugoniot conditions. If for
a given model ($W$ function) such solution does not satisfy the entropy
condition, we can consider the model as inadequate to describe the
compression shock.

In this work we give answers to all mentioned above questions. The  obtained
results do not appear in the reviewed literature.

It has to be added also that the concept of a weak solution is well known in
the literature. For example in \cite{EVA} one can find a definition of a
weak solution of an initial value problem for a system of conservation laws
in two variables. Using a general idea of that concept we define what it
means to be a weak solution of an initial and boundary value problem for
p-systems. This definition does not appear explicitly in the reviewed
literature.



The paper is organized as follows: In Section 2 the main notation and
concepts are introduced: conservation laws, hyperbolic system,
weak solution, Rankine-Hugoniot condition, genuine nonlinearity,
entropy/entropy-flux pair. Next, we give a brief presentation of basic
concepts of the theory of elasticity, such as, deformation gradient,
deformation tensor, second Piola-Kirchhoff stress tensor and first Piola-tensor.
 We also present four forms of $W$ (strain-energy function) appearing in the theory
of elasticity, to model a behavior of certain materials. We refer to them
as: St.Venant-Kirchhoff, Kirchhoff modified, Ogden and Blatz-Ko-Ogden
functions.

In Section 3 we consider one dimensional reduction of the system of
partial differential equations for elasticity, which depends on the
strain-energy function $W$ and results in a p-system. Also, we introduce the
notions of hyperbolicity, no interpenetration of matter and genuine
nonlinearity.

In Section 4
 we provide the concept of weak solutions for various
versions of an IBVP (initial and boundary value problem) for a p-system,
including a particular case of IBVP, $IBVP_{V_0}$, and we find its solutions
employing the Rankine-Hugoniot conditions, we denote such solution by $S(V_0)
$.

In Section 5 we discuss the notions of an entropy/entropy-flux
pair for a p-system, entropy condition, entropy condition for a solution of 
$IBVP_{V_0}$ and standard entropy function. We also establish the importance
of the requirements of hyperbolicity(strict) and genuine nonlinearity, as
being essential in proving if a weak solution is an entropy solution.

In Section 6  we show the results concerning
to hyperbolicity and genuine nonlinearity for the models under consideration
and the entropy condition corresponding to a standard entropy function for a
solution of $IBVP_{V_0}$. 

Finally, in Section 7 we present a summary of the main conclusions of our research.

\section{Preliminaries}
\subsection{Conservation laws and related concepts}

We begin this section with some essential definitions, that we will use in
the course of this work.%

A conservation law asserts that the change in the total amount of a physical
entity contained in any bounded region $G\subset \mathbb{R}^n $ of space is
due to the flux of that entity across the boundary of $G$. In particular,
the rate of change is
\begin{equation}  \label{c-law}
\frac{d}{dt}\int_G \mathbf{u}dX=-\int_{\partial G}\mathbf{F(u)n} dS,
\end{equation}

where $\mathbf{u}=\mathbf{u}(X,t)=(u^1(X,t),\ldots ,u^m(X,t) )\ (X\in
\mathbb{R}^n, t\geq 0)$ measures the density of the physical entity under
discussion, the vector \\
$\mathbf{F}:\mathbb{R}^m\rightarrow \mathbb{M}%
^{m\times n}$ describes its flux and $\mathbf{n}$ is the outward normal to
the boundary $\partial G$ of $G$. Here $\mathbf{u}$ and $\mathbf{F}$ are $C^1
$ functions. Rewriting (\ref{c-law}), we deduce

\begin{equation}  \label{ec-div}
\int_G \mathbf{u}_t dX=-\int_{\partial G}\mathbf{F(u)n}dS=-\int_G div
\mathbf{F(u)}dX.
\end{equation}

As the region $G\subset \mathbb{R}^n$ was arbitrary, we derive from (\ref%
{ec-div}) this initial-value problem for a general system of conservation
laws:
\begin{equation}
\left\{ \begin{aligned} \mathbf{u}_t +div \mathbf{F(u)}&=0 &\text{in}&\
\mathbb{R}^n \times (0,\infty) \\ \mathbf{u}&=g &\text{on}&\ \mathbb{R}^n
\times \{t=0\} \end{aligned} \right.
\end{equation}

where $g=(g^1,\ldots, g^m)$ is a given function describing the initial
distribution of $\mathbf{u}=(u^1,\ldots ,u^m).$ In particular, the
initial-value problem for a system of conservation laws in one-dimensional
space, takes the following form


\begin{equation}  \label{c-law2}
\mathbf{u}_t + \mathbf{F(u)}_X=0 \ \ \ \text{in}\ \mathbb{R} \times
(0,\infty)
\end{equation}
with initial condition given by
\begin{equation}  \label{IC}
\mathbf{u}(X,t)=g \ \ \ \text{on}\ \mathbb{R} \times \{t=0\}
\end{equation}
where $\mathbf{F}:\mathbb{R}^m\rightarrow \mathbb{R}^m $ and $g:\mathbb{R}
\rightarrow \mathbb{R}^m$ are given and $\mathbf{u}:\mathbb{R}\times
[0,\infty)\rightarrow \mathbb{R}^m$ is the unknown, $\mathbf{u}=\mathbf{u}%
(X,t)$ \cite{EVA}.

For $C^1$ functions the conservation law (\ref{c-law2}) is equivalent to
\begin{equation}  \label{matrix-law}
\mathbf{u}_t+\mathbf{B(u)}\mathbf{u}_X=0 \ \ \ \text{in}\ \mathbb{R} \times
(0,\infty)
\end{equation}
where $\mathbf{B}:\mathbb{R}^m\rightarrow \mathbb{M}^{m\times m} $ is given
by $B(z)=DF(z)$, for \\
$z=(z_1,\ldots, z_m) \in \mathbb{R}^m$, where
\begin{equation}  \label{matrix}
D\mathbf{F}(z)= \left(
\begin{array}{ccc}
F^1_{z_1} & \cdots & F^1_{z_m} \\
\vdots & \ddots & \vdots \\
F^m_{z_1} & \cdots & F^m_{z_m}%
\end{array}
\right).
\end{equation}

 If for each $z \in \mathbb{R}^m$ the eigenvalues of $\mathbf{B}%
(z)$ are real and distinct, we call the system (\ref{matrix-law}) \textbf{strictly
hyperbolic} \cite{EVA}.

 A system of conservation laws (\ref{matrix-law})  is said to be
\textbf{genuinely nonlinear} in a region $\Omega \subseteq \mathbb{R}^{n}$ if
\begin{equation*}
\nabla \lambda_k \cdot \mathbf{r}_k \neq 0,
\end{equation*}
for $k=1,2,\ldots,n $ at all points in $\Omega$, where $\lambda_k(z)$ are
the eigenvalues of $\mathbf{B}(z)$, with corresponding eigenvectors $\mathbf{%
r}_k(z)$ \cite{RINAR}.

\begin{definition}
The p-system is a conservation law being this collection of two equations:
\begin{equation}
\left\{ \begin{aligned} u^2_t - p(u^1)_X&=0 \ \text{(Newton's law)} \\
u^1_t-u^2_X &=0 \ \text{(compatibility condition)} \end{aligned} \right.
\end{equation}
in $\mathbb{R}\times (0,\infty)$, where $p:\mathbb{R}\rightarrow \mathbb{R}$
is given. Here $F(z)=(-p(z_1),-z_2)$ for $z=(z_1,z_2)$\ \cite{EVA}.
\end{definition}

\begin{definition}
A weak solution of (\ref{c-law2}) is a function $\mathbf{u}\in L^\infty(%
\mathbb{R} \times (0,\infty); \mathbb{R}^m)$ such that
\begin{equation*}
\int_0^\infty \int_{-\infty}^{\infty} (\mathbf{u}\cdot \mathbf{\phi}_t +%
\mathbf{F(u)}\cdot \mathbf{\phi}_X) dXdt + \int_{-\infty}^{\infty} (\mathbf{g%
}\cdot \mathbf{\phi})|_{t=0} \ dX =0
\end{equation*}

for every smooth $\phi:\mathbb{R}\times [0,\infty)\rightarrow \mathbb{R}^m$,
with compact support \ \cite{EVA}.
\end{definition}








\subsection{Basic notions of Elasticity Theory}

We consider a continuous body which occupies a connected open subset of a
three-dimensional Euclidean point space, and we refer to such a subset as a
configuration of the body. We identify an arbitrary configuration as a
reference configuration and denote this by $\mathfrak{B_{0}}$. Let points in
$\mathfrak{B_{0}}$ be labelled by their position vectors $\mathbf{X}%
=(X^{1},X^{2},X^{3})$, where $X^{1},X^{2}$ and $X^{3}$ are coordinates
relative to an arbitrary chosen Cartesian orthogonal coordinate system. Now
suppose that the body is deformed from $\mathfrak{B_{0}}$ so that it
occupies a new configuration, which is denoted by $\mathfrak{B_{t}}.$ We
refer to $\mathfrak{B_{t}}$ as the deformed configuration of the body. The
deformation is represented by the mapping $\phi _{t}:\mathfrak{B_{0}}%
\rightarrow \mathfrak{B_{t}}$ which takes points $\mathbf{X}$ in $\mathfrak{%
B_{0}}$ to points $\mathbf{x}=(x_{1},x_{2},x_{3})$ in $\mathfrak{B_{t}}$,
where $x_{1},x_{2}$ and $x_{3}$ are coordinates relative to the same
Cartesian orthogonal coordinate system as $X^{1},X^{2}$ and $X^{3}$. Thus,
the position vector of the point $\mathbf{X}$ in $\mathfrak{B_{t}}$, which
is denoted by $\mathbf{x}$, is
\begin{equation*}
\mathbf{x=\phi (X,t)\equiv \phi _{t}(X)}.
\end{equation*}

The mapping $\phi$ is called the deformation from $\mathfrak{B_0}$ to $%
\mathfrak{B_t}$. We require $\phi_t$ to be sufficiently smooth, orientation
preserving and invertible. The last two requirements mean physically, that
no interpenetration of matter occurs.

\subsubsection{Deformation gradient, deformation tensor, strain-energy function
and time evolution of an elastic body}

Now, we introduce some basic definitions of Elasticity theory, namely:
deformation gradient, deformation tensor, second Piola-Kirchhoff stress
tensor, first Piola-tensor \cite{MAR}. We restrict our discussion to
hyperelastic, homogeneous and isotropic materials.

\begin{itemize}

\item 
$F^{a}_A(X,t)=\frac{\partial \phi^{a}}{\partial X^{A}}$ \ \ (Deformation gradient ),
 $a,A\in \{1,2,3\}$.

\item $C=F^{T}F$ or componentwise  by $C_{AB}=\delta_{ij}F_A^iF_B^j$, 
$A,B\in \{1,2,3\}$. \ \ (Deformation tensor).
\item Principal invariants of $C$:

$I_1=\text{tr}(C),
I_2=(\det(C))\text{tr}(C^{-1}),
I_3(C)=\det(C).$

\item (Second Piola-Kirchhoff stress tensor)
\begin{equation*}
S^{AB}=2\bigg\{\frac{\partial W}{\partial I_{1}}G^{AB}+\bigg(\frac{\partial W%
}{\partial I_{2}}I_{2}+\frac{\partial W}{\partial I_{3}}I_{3}\bigg)C^{-1}-%
\frac{\partial W}{\partial I_{2}}I_{3}C^{-2}\bigg\},
\end{equation*}

where $G^{AB}$ is Kronecker's delta and $W$ is the strain-energy function.

\item (The first Piola-tensor)
$$P^{iA}=F^i_B S^{BA}=F^i_1 S^{1A}+F^i_2 S^{2A}+F^i_3 S^{3A},\
\text{where} \ i,A, B\in \{1,2,3\}.$$
\end{itemize}

We consider the following four forms of $W,$ \cite{MAR}:

\begin{enumerate}
\item St.Venant-Kirchhoff

\begin{equation}  \label{Kir}
W=\frac{\lambda}{8}(I_1-3)^2+\frac{\mu}{4}(I_1^2-2I_2-2I_1+3).
\end{equation}

\item Kirchhoff modified
\begin{equation}  \label{KirM}
W=\frac{\lambda}{8}(\ln I_3)^2+\frac{\mu}{4}(I_1^2-2I_2-2I_1+3).
\end{equation}

\item Ogden
\begin{equation}  \label{Ogden}
W=\frac{\mu}{2}\big(I_1-3-2\ln(\sqrt{I_3})\big)+\frac{\lambda}{2}(\sqrt{I_3}-1)^2.
\end{equation}

\item Blatz-Ko-Ogden
\begin{equation}
W=f\frac{\mu }{2}[(I_{1}-3)+\frac{1}{\beta }(I_{3}^{-\beta }-1)]+(1-f)\frac{%
\mu }{2}\bigg[\frac{I_{2}}{I_{3}}-3+\frac{1}{\beta }(I_{3}^{\beta }-1)\bigg].
\label{B-KO}
\end{equation}
\end{enumerate}

We can see that the functions (\ref{Kir})-(\ref{Ogden}) depend on two
parameters: Lam\'e moduli $\lambda$ and $\mu$, where $\lambda,\mu >0 $. In (%
\ref{B-KO}) $\beta=\frac{\lambda}{2\mu} $ and this $W$ depends also on a
parameter $f$ restricted by $0<f<1.$

Finally, the components of the mapping
\begin{equation*}
\phi(\mathbf{X},t)=(\phi^1(\mathbf{X},t),\phi^2(\mathbf{X},t), \phi^3(%
\mathbf{X},t))
\end{equation*}
are subject to the following system of PDE's, describing the evolution of an
elastic body:
\begin{equation}  \label{ec-nonlinear}
\rho_0 \frac{\partial^2 \phi^i}{\partial t^2}=\frac{\partial P^{iA}}{%
\partial X^A}.
\end{equation}
Here $\rho_0=\rho_0(\mathbf{X})$ is the mass density in reference
configuration assumed further to be constant.



\section{One-dimensional reduction for certain models of elastic materials}

In this section we present the reduction to the one-dimensional case, 
which we will maintain in all the paper. Also, we rewrite the
requirements of: hyperbolicity, no interpenetration of matter and genuine nonlinearity,
to the one-dimensional case.   

 We assume that there is a motion of particles only in the direction of   $X^1$-axis, that is:
\begin{equation}\label{reduction}
\begin{cases}
\phi^1(\mathbf{X},t)=X^1+U(X^1,t)\\
\phi^2(\mathbf{X},t)=X^2\\
\phi^3(\mathbf{X},t)=X^3.
\end{cases}
\end{equation}

Then  $F^i_{A}, C_{AB}, C^{-1}_{AB}, I_1, I_2 $ and $I_3$ become

\begin{equation}
F^{i}_{A}=
\left(
  \begin{array}{ccc}
    \frac{\partial \phi^1}{\partial X^1} & \frac{\partial \phi^1}{\partial X^{2}} & \frac{\partial \phi^1}{\partial X^{3}} \\\\
    \frac{\partial \phi^2}{\partial X^1} & \frac{\partial \phi^2}{\partial X^2} & \frac{\partial \phi^2}{\partial X^3} \\\\
    \frac{\partial \phi^3}{\partial X^1} & \frac{\partial \phi^3}{\partial X^2}& \frac{\partial \phi^3}{\partial X^3}\\
   \end{array}
 \right)=
 \left(
   \begin{array}{ccc}
     \frac{\partial \phi^1}{\partial X^1} & 0 & 0 \\\\
     0 & 1 & 0 \\\\
     0 & 0& 1 \\
   \end{array}
 \right).
\end{equation}

\begin{equation}
 C_{AB}=
 \left(
   \begin{array}{ccc}
     (F^1_{1})^2 & 0 & 0 \\
     0 & 1 & 0 \\
     0 & 0 & 1 \\
  \end{array}
 \right), 
 C_{AB}^{-1}=\left(
          \begin{array}{ccc}
            1/(F^{1}_{ 1})^2 & 0 & 0 \\
            0 & 1 & 0 \\
            0 & 0 & 1\\
          \end{array}
       \right).
\end{equation}

\begin{equation*}
I_1=2+(F^1_{1})^2,
   I_2=2(F^1_{1})^2+1,
I_3=(F^1_{1})^2.
\end{equation*}

Therefore   the  system (\ref{ec-nonlinear}) becomes

 $$\rho_0 \frac{\partial^2 \phi^i}{\partial t^2}=\frac{\partial (P^{iA})}{\partial X^A}.$$

More specifically,
\begin{equation} \label{ec-03-02}
 \begin{split}
 \rho_0 \frac{\partial^2 \phi^1}{\partial t}&= P{,}^{11}_{1}+P{,}^{12}_{2}+P{,}^{13}_{3},\\
 \rho_0 \frac{\partial^2 \phi^2}{\partial t}&= P{,}^{21}_{1}+P{,}^{22}_{2}+P{,}^{23}_{3},\\
 \rho_0 \frac{\partial^2 \phi^3}{\partial t}&= P{,}^{31}_{1}+P{,}^{32}_{2}+P{,}^{33}_{3}.
 \end{split}
\end{equation}

Notice that $P^{11}=F^1_1S^{11}+{F^1_2} S^{21}+F^1_3S^{31}=F^1_1S^{11},$
$P^{12}=P^{21}=P^{13}=P^{31}=P^{23}=P^{32}=0, P^{22}=S^{22},P^{33}=S^{33},$
and $ \frac{\partial \phi^1}{\partial t}=\frac{\partial U}{\partial t},
 \ \frac{\partial \phi ^2}{\partial t}=0
  = \frac{\partial \phi ^3}{\partial t}.$
Consequently  (\ref{ec-03-02}) is reduced to one equation, which after denoting $X^1$ by $X$ and putting $P=\frac{P^{11}}{\rho_0},$ reads

\begin{equation}\label{ec-nonlinearp}
\frac{\partial^2 U}{\partial t^2}=\frac{\partial P}{\partial X}.
\end{equation}

Setting $V=\frac{\partial U}{\partial t}$ and $\Gamma=\frac{\partial U}{\partial X}$ one obtains a p-system of first order PDE's:

\begin{equation} \label{ec-p}
\left\{
  \begin{array}{ll}
    V_t - (P(\Gamma))_X&=0 \\
    \Gamma_t-V_X &=0
  \end{array}
\right.
\end{equation}

\begin{remark} \label{matter}
Under the assumption  (\ref{reduction}) the requirement of no interpenetration of matter means that $\phi_X>0$, i.e., $1+\mathbf{U}_X>0$.
\end{remark}

Notice that, the p-system (\ref{ec-p}) can be rewritten as

\begin{equation}\label{ec-p2}
\mathbf{u}_t+\mathbf{B(u)}\mathbf{u}_X=0
\end{equation}

where $\mathbf{u}=(V,\Gamma)$ and
$\mathbf{B}=
\left(
   \begin{array}{cc}
     0 & -P'(\Gamma) \\
     -1 & 0 \\
   \end{array}
 \right).$
 The eigenvalues of $\mathbf{B}$ are  $\lambda_1=-\sqrt{P'(\Gamma)}$ and $\lambda_2= \sqrt{P'(\Gamma)}$ with corresponding eigenvectors
$\mathbf{r}_1=(\sqrt{P'(\Gamma)},1)$ and $\mathbf{r}_2=(-\sqrt{P'(\Gamma)},1)$.

\begin{remark}
Note that for our case of a p-system, no interpenetration of matter condition, $\phi_X>0$, is equivalent to  $\Gamma>-1$, since $\phi_X=1+U_X.$
\end{remark}

\begin{remark}

\begin{itemize}
  \item The p-system (\ref{ec-p2}) is  \textbf{strictly hyperbolic} if $P'>0$, everywhere in the domain of  $P(\Gamma).$
  \item The p-system (\ref{ec-p2}) is \textbf{genuinely nonlinear} in a region $\Omega$ of the domain of $P(\Gamma)$ if $P''\neq0$ everywhere in $\Omega$.

Indeed, it is so since
 $-\nabla \lambda_1 \cdot \mathbf{r}_1=\nabla \lambda_2 \cdot \mathbf{r}_2=\frac{P''(\Gamma)}{2\sqrt{P'(\Gamma)}} .$

 By continuity of $P''(\Gamma)$, genuine nonlinearity means that $P''(\Gamma)$ is of constant sign in $\Omega$.
 However we will call a  p-system  (\ref{ec-p2})  genuinely nonlinear if  $P''<0$, since this requirement  plays  an important role in studying entropy inequality.
\end{itemize}
\end{remark}
We remark also that  hyperbolicity condition is an essential physical  requirement, since it guarantees  that particles have a finite propagation
speed. 
Now, we obtain explicit forms of the function $P$ for the 
models under consideration. Indeed,







{St.Venant-Kirchhoff:}
  $P(\Gamma)=\big(\frac{\lambda+2\mu}{2\rho_0}\big)(1+\Gamma)(2+\Gamma)\Gamma.$

{Modified Kirchhoff:}
 $P(\Gamma)=\frac{1}{\rho_0}\bigg(\mu(1+\Gamma)^3-\mu(1+\Gamma)+\lambda \frac{\ln(1+\Gamma)}{(1+\Gamma)} \bigg).$

{Ogden:}
$P(\Gamma)=\frac{1}{\rho_0}\big(\lambda \Gamma+\mu \frac{(2+\Gamma)\Gamma}{\Gamma+1}\big).$

{Blatz-Ko and Ogden:}
\begin{equation}\label{p4}
P(\Gamma)=\frac{\mu(1+\Gamma)}{\rho_0}\Bigg\{f \bigg[1-(1+\Gamma)^{-2\beta-2}\bigg]+\frac{(1-f)}{(1+\Gamma)^4}\bigg[(1+\Gamma)^{2\beta+2}-1\bigg]\Bigg\}.
\end{equation}


\begin{definition}
If $P(\Gamma)=\frac{(\lambda +2\mu)}{\rho_0}\Gamma$, the model is called linear model.
\end{definition}


%

\section{Weak solution of an IBVP for a p-system}

In this section we give the concept of  weak solutions for various versions of an IBVP (initial and boundary value problem), for a p-system,
including a particular case of IBVP, $IBVP_{V_0}$.\\
We also provide notions  of an entropy/entropy-flux pair  and entropy condition for a solution of $IBVP_{V_0}$.

 Our aim is to give an answer to the question about a weak solution  for an IVBP for (\ref{ec-p}) with these initial
and boundary conditions:
\begin{equation}\label{ec-m1}
\begin{cases}
V(X,0)=f(X)   \\
 \Gamma(X,0)=g(X)\\
 P(\Gamma(0,t))+a(t)V(0,t)=c(t)
 \end{cases}
\end{equation}
 or
 \begin{equation}\label{ec-m2}
\begin{cases}
V(X,0)=f(X)   \\
 \Gamma(X,0)=g(X)\\
 V(0,t)+b(t)P(\Gamma(0,t))=c(t).
 \end{cases}
\end{equation}

To define a weak solution of such IBVP   in the first quadrant of the $Xt$-plane, we use arbitrary $C^1$ functions $\varphi$, $\psi$ and $\chi$,
$$\varphi,\psi, \chi :[0,\infty)\times[0,\infty)\rightarrow \mathbb{R}$$
of compact supports. We refer to those functions as test functions.

\begin{proposition}\label{propo2}
Let $f, g,a$ and $c$ be $C^1$ functions  on $[0,\infty)$, and let $V(X,t)$, $\Gamma(X,t)$
 be $C^1$ functions   on $[0,\infty)^2$, such that $P(\Gamma(X,t))$  is $C^1$ on $([0,\infty)^2)$. Then the pair $(V,\Gamma)$ is a classical solution
  of IBVP (\ref{ec-p}),(\ref{ec-m1}), if and only if for all $\varphi$ and $\psi$, with $\psi$  satisfying the condition  $\psi(0,t)=0$,
    it holds

\begin{equation} \label{ws3}
-\int_0^{\infty} g(X)\psi(X,0)dX -\int_0 ^{\infty} \int_0 ^{\infty}\Gamma \psi_t dtdX
+\int_0 ^{\infty}\int_0 ^{\infty}V \psi_X dXdt=0
\end{equation}
and
\begin{multline}\label{ws4}
-\int_0^{\infty} f(X)\varphi(X,0)dX- \int_0^{\infty}\int_0^{\infty} V\varphi_t dtdX+  \int_0^{\infty}c(t)\varphi(0,t)dt
 \\+  \int_0^{\infty}\int_0^{\infty} P(\Gamma)\varphi_X dXdt-
\int_0^{\infty} g(X)a(0)\varphi(X,0)dX \\ -\int_0^{\infty}\int_0^{\infty}  (a(t)\varphi(X,t))_t \Gamma dtdX
+ \int_0^{\infty}\int_0^{\infty}a(t)\varphi_X(X,t)V dXdt=0.
\end{multline}

\end{proposition}
\begin{proof}

Indeed, assuming that $(V,\Gamma)$ is a classical solution of IBVP (\ref{ec-p}),(\ref{ec-m1}), we
multiply  the first equation in (\ref{ec-p}) by  $\varphi$, integrating  by parts
and using the initial and boundary conditions (\ref{ec-m1}) we obtain 
 


\begin{equation}\label{ec12}
\begin{split}
-&\int_0^{\infty} f(X)\varphi(X,0)dX-\int_0^{\infty}\int_0^{\infty} V(X,t)\varphi_t(X,t)dtdX+\int_0^{\infty}c(t)\varphi(0,t)dt
\\ -&\int_0^{\infty}a(t)V(0,t)\varphi(0,t)dt+
\int_0^{\infty}\int_0^{\infty} P(\Gamma)\varphi_X(X,t)dXdt=0.
\end{split}
\end{equation}
Similarly multiplying the second equation in (\ref{ec-p}) by a test function  $\chi$ and integrating by parts results in
\begin{multline}\label{ec13}
-\int_0^{\infty} g(X)\chi(X,0)dX -\int_0 ^{\infty} \int_0 ^{\infty}\Gamma(X,t)\chi_t(X,t)dtdX +\int_0^{\infty} V(0,t) \chi(0,t)dt\\
+\int_0 ^{\infty}\int_0 ^{\infty}V(X,t)\chi_X(X,t)dXdt=0.
\end{multline}
For    $\chi=\psi$ the equation  (\ref{ec13}) is equivalent to
\begin{multline*} 
-\int_0^{\infty} g(X)\psi(X,0)dX -\int_0 ^{\infty} \int_0 ^{\infty}\Gamma(X,t)\psi_t(X,t)dtdX \\
+\int_0 ^{\infty}\int_0 ^{\infty}V(X,t)\psi_X(X,t)dXdt=0,
\end{multline*}
 which is (\ref{ws3}).\\
 Next, assuming that  $\chi(X,t)=a(t)\varphi(X,t)$  the equation (\ref{ec13}) becomes
\begin{multline}\label{ec15}
-\int_0^{\infty} g(X)a(0)\varphi(X,0)dX-\int_0^{\infty}\int_0^{\infty} \Gamma(X,t)(a(t)\varphi(X,t))_t dtdX+\\
\int_0^{\infty}a(t) V(0,t)\varphi(0,t)dt+\int_0^{\infty}\int_0^{\infty}V(X,t)(a(t)\varphi(X,t))_X dXdt=0.
\end{multline}
Now, adding  (\ref{ec12}) to (\ref{ec15}), we get
\begin{multline*}
-\int_0^{\infty} f(X)\varphi(X,0)dx- \int_0^{\infty}\int_0^{\infty} V(X,t)\varphi_t(X,t) dtdX+  \int_0^{\infty}c(t)\varphi(0,t)dt
\\ +  \int_0^{\infty}\int_0^{\infty} P(\Gamma)\varphi_X(X,t)dXdt -
\int_0^{\infty} g(X)a(0)\varphi(X,0)dX \\ -\int_0^{\infty}\int_0^{\infty} \Gamma(X,t)(a(t)\varphi(X,t))_t dtdX
 + \int_0^{\infty}\int_0^{\infty}V(X,t)(a(t)\varphi(X,t))_X dXdt=0,
\end{multline*}
which is (\ref{ws4}).\\
Next,  it remains to verify  that if $(V,\Gamma)$ satisfies  (\ref{ws3}) and (\ref{ws4}) for all $\varphi$ and $\psi$, then
$(V,\Gamma)$ is a classical solution of the IBVP  (\ref{ec-p}), (\ref{ec-m1}).\\
Indeed, integrating by parts the equation (\ref{ws3}) we obtain
\begin{equation}\label{ec17}
\int_0^{\infty}\psi(X,0)(\Gamma(X,0)-g(X))dX+\int_0^{\infty}\int_0^{\infty} \psi(X,t)(\Gamma_t(X,t)-V_X(X,t))dXdt=0.
\end{equation}
If in addition
 $\psi$ has compact support in $ (0,\infty)\times (0,\infty)$, we obtain
\[\int_0^{\infty}\int_0^{\infty} \psi(X,t)(\Gamma_t(X,t)-V_X(X,t))dXdt=0,\]
 for all such test functions $\psi$. Therefore we conclude
\begin{equation} \label{ec19}
\Gamma_t(X,t)-V_X(X,t)=0.
\end{equation}
Now since  $\Gamma_t(X,t)-V_X(X,t)=0$, then  (\ref{ec17}) reduces to
\[\int_0^{\infty}\psi(X,0)(\Gamma(X,0)-g(X))dX=0, \]
and this holds for all function  $\psi$, with compact support in $[0,\infty)\times [0,\infty)$, containing points on the $X$-axis, and subject to
$\psi(0,t)=0$ we get 
\begin{equation}\label{ec20}
\Gamma(X,0)=g(X).
\end{equation}
Similarly integrating by parts the equation  (\ref{ws4}), we obtain
\begin{multline}\label{ec18}
\int_0^{\infty}\varphi(X,0)(V(X,0)-f(X))dX+\int_0^{\infty}\varphi(0,t)(c(t)-P(\Gamma(0,t))-V(0,t)a(t))dt\\+
\int_0^{\infty}\int_0^{\infty} \varphi(X,t)(V_t(X,t)-(P(\Gamma))_X)dXdt + \int_0^{\infty}a(0)\varphi(X,0)(\Gamma(X,0)-g(X))dX\\
+ \int_0^{\infty}\int_0^{\infty}a(t)\varphi(X,t)(\Gamma_t(X,t)-V_X(X,t))dXdt=0,
\end{multline}
which because of   (\ref{ec19}) and (\ref{ec20})  becomes
\begin{multline}\label{ec21}
\int_0^{\infty}\varphi(X,0)(V(X,0)-f(X))dX+\int_0^{\infty}\varphi(0,t)(c(t)-P(\Gamma(0,t))-V(0,t)a(t))dt\\+
\int_0^{\infty}\int_0^{\infty} \varphi(X,t)(V_t(X,t)-(P(\Gamma))_X)dXdt=0.
\end{multline}
If in addition  $\varphi$ has compact support in  $(0,\infty)\times (0,\infty)$, we obtain
\[\int_0^{\infty}\int_0^{\infty} \varphi(X,t)(V_t(X,t)-(P(\Gamma))_X)dXdt=0\]
for all such test functions $\varphi$. Therefore  we conclude that 
$V_t(X,t)-(P(\Gamma))_X=0.$
Now if  $V_t(X,t)-(P(\Gamma))_X=0$, then  the equation (\ref{ec21}) becomes

\begin{equation}\label{ec22}
\int_0^{\infty}\varphi(X,0)(V(X,0)-f(X))dX+\int_0^{\infty}\varphi(0,t)(c(t)-P(\Gamma(0,t))-V(0,t)a(t))dt=0.
\end{equation}
Assuming that $\varphi$ has compact support in $[0,\infty)\times [0,\infty)$ containing points on the  $X$-axis, but not on
the $t$-axis, we get $\int_0^{\infty}\varphi(X,0)(V(X,0)-f(X))dX=0,$
thus $V(X,0)-f(X)=0,$
then   (\ref{ec22}) reduces to
\[ \int_0^{\infty}\varphi(0,t)(c(t)-P(\Gamma(0,t))-V(0,t)a(t))dt=0.\]
Since this holds for all function $\varphi$ with compact support in  $[0,\infty)\times [0,\infty)$, and  containing points on the  $t$-axis, then it follows
\[P(\Gamma(0,t))+a(t)V(0,t)=c(t).\]
Therefore we conclude that  $(V,\Gamma)$ is a classical solution for the IVBP  (\ref{ec-p}), (\ref{ec-m1}).
\end{proof}
\textit{Proposition} \ref{propo2} suggests the following definition:
\begin{definition}
Let $f,g$ and $c \in L^{\infty}([0,\infty))$, and $a \in C^1([0,\infty))$. We say that the pair $(V,\Gamma)\in L^\infty ([0,\infty)^2) $, such that
$P(\Gamma(X,t)) \in L^{\infty}([0,\infty)^2)$, is a weak solution
of IBVP (\ref{ec-p}),(\ref{ec-m1}), provided  (\ref{ws3}) and (\ref{ws4}) hold for all test
functions $\varphi$ and $\psi$ with $\psi$ restricted by $\psi(0,t)=0.$
\end{definition}

Next, similarly we have an  proposition analogous to the \textit{Proposition }\ref{propo2}, for the  IBVP  (\ref{ec-p}),(\ref{ec-m2}).
\begin{proposition}\label{propo3}
Let $f, g,b$ and $c$ be $C^1$ functions  on $[0,\infty)$, and let $V(X,t)$, $\Gamma(X,t)$ be $C^1$ functions   on
$[0,\infty)^2$, such that $P(\Gamma(X,t))$ is $C^1$ on $[0,\infty)^2$. Then the pair $(V,\Gamma)$ is a classical solution of IBVP (\ref{ec-p}),(\ref{ec-m2}), if and only if,
for all  $\varphi$ and $\psi$, with $\varphi$ satisfying the condition
 $\varphi(0,t)=0$ it holds

\begin{multline} \label{ws5}
-\int_0^{\infty} f(X)\varphi(X,0)dX-
\int_0^{\infty}\int_0^{\infty} V(X,t)\varphi_t(X,t)dtdX\\
+\int_0^{\infty}\int_0^{\infty} P(\Gamma)\varphi_X(X,t)dXdt=0
\end{multline}
and
\begin{multline}\label{ws6}
-\int_0^{\infty} g(X)\psi(X,0)dX -\int_0 ^{\infty} \int_0 ^{\infty}\Gamma(X,t)\psi_t(X,t)dtdX +\int_0^{\infty}c(t)\psi(0,t)dt\\+
\int_0^{\infty}\int_0^{\infty}V(X,t)\psi_X(X,t)dXdt -\int_0^{\infty}f(X)b(0)\psi(X,0)dX\\- \int_0^{\infty}\int_0^{\infty}(b(t)\psi(X,t))_t V dtdX
+\int_0^{\infty}\int_0^{\infty}b(t)\psi_X P(\Gamma) dXdt=0.
\end{multline}

\end{proposition}

Proposition \ref{propo3}, suggests the following definition:
\begin{definition}
Let $f,g, c \in L^{\infty}([0,\infty))$, and $b\in C^1([0,\infty))$. We say that the pair $(V,\Gamma)\in L^\infty ([0,\infty)^2)$, such that
$P(\Gamma(X,t)) \in L^{\infty}([0,\infty)^2)$, is a weak solution
of the IBVP (\ref{ec-p}),(\ref{ec-m2}), provided  (\ref{ws5}) and (\ref{ws6}) hold for all test
functions $\varphi$ and $\psi$ with $\varphi$ restricted by  $\varphi(0,t)=0$.
\end{definition}

\begin{remark}
Note that setting $b(t)=0$ and $c(t)=h(t)$ in  (\ref{ws6}) we obtain  (\ref{ws2}). This is consistent  with the
fact that the initial and boundary conditions (\ref{ec-m2}) are equivalent to (\ref{ec3-2}) for this choice of $b$ and $c$.
\end{remark}

%
Now, we consider a p-system  (\ref{ec-p}), with the following initial and boundary conditions:
\begin{equation}\label{ec3-2}
\begin{cases}
V(X,0)=f(X)   \\
 \Gamma(X,0)=g(X)\\
 V(0,t)=h(t).
\end{cases}
\end{equation}
\
\

Notice that (\ref{ec3-2}) is a particular case of the condition (\ref{ec-m2}).


\begin{proposition} \label{propo1}
Let $f, g$ and $h$ be $C^1$ functions  on $[0,\infty)$, and let $V(X,t)$, $\Gamma(X,t)$
 be $C^1$ functions   on
$[0,\infty)^2$, such that $P(\Gamma(X,t))$  is $C^1$ on $[0,\infty)^2$. Then  the pair $(V,\Gamma)$ is a classical solution
of IBVP (\ref{ec-p}), (\ref{ec3-2}), if and only if for all $\varphi$ and $\psi$, with $\varphi$ satisfying the condition   $\varphi(0,t)=0,$ it holds

\begin{equation}\label{ws1}
\int_0^\infty \int_0^\infty (V\varphi_t-P(\Gamma)\varphi_X)dtdX+\int_0^\infty f(X)\varphi(X,0)dX=0
\end{equation}

and
\begin{equation}\label{ws2}
\int_0^\infty \int_0^\infty (\Gamma\psi_t-V\psi_X)dtdX-\int_0^\infty h(t)\psi(0,t)dt+\int_0^\infty g(X)\psi(X,0)dt=0.
\end{equation}
\end{proposition}

\textit{Proposition} \ref{propo1}, suggests the following definition:
\begin{definition}\label{def-propo1}
Let $f,g,h \in L^{\infty}([0,\infty))$.
We say that the pair 
$(V,\Gamma) \in L^\infty ([0,\infty)^2)$, such that $P(\Gamma(X,t)) \in L^{\infty}([0,\infty)^2)$, is a weak solution
of IBVP (\ref{ec-p}), (\ref{ec3-2}), provided  (\ref{ws1}) and (\ref{ws2}) hold for all test
functions $\varphi$ and $\psi$, with  $\varphi$ restricted by $\varphi(0,t)=0$.
\end{definition}
\begin{remark}
A interesting question arises in what sense $(V,\Gamma)$ satisfies (\ref{ec3-2}). To answer that question we need to prove that the traces, \cite{RINAR}, \cite{EVA},
 of $(V,\Gamma)$ on the positive part of the $X$ axes and of $V$ on the positive part of the $t$ axes exist and  are equal
 to $f(X)$, $g(X)$ and $h(t)$ respectively. That problem seems to be non trivial. Its solution does not appear in the revised literature.
\end{remark}

\begin{remark}
A similar argument shows that by replacing $\varphi(0,t)=0$ by $\psi(0,t)=0$, in the  \textit{Definition} \ref{def-propo1}, we can arrive to an analogous definition of  a weak solution
 of the system (\ref{ec-p}), with the following initial and boundary conditions
\begin{equation}\label{ec8}
\begin{cases}
V(X,0)=f(X)   \\
 \Gamma(X,0)=g(X)\\
 P(\Gamma(0,t))=h(t).
\end{cases}
\end{equation}
\end{remark}



\subsection{An initial and  boundary value problem, $IBVP_{V_0}$, for a p-system}

  We consider a particular case of an  IBVP, (\ref{ec-p}), (\ref{ec3-2}),  denoted further by $IBVP_{V_{0}}.$

\begin{equation} \label{ec281}
\begin{cases}
V(X,0)=-V_0 \ \   \\
 \Gamma(X,0)=0\\
 V(0,t)=0.
\end{cases}
\end{equation}

\begin{theorem}\label{theoremA}

The pair $(V,\Gamma)$, given by the equations (\ref{ec-0206})
and (\ref{ec-0207}), is a weak solution of $IBVP_{V_0}$
(\ref{ec-p}),(\ref{ec281}),
\begin{equation} \label{ec-0206}
V(X,t)=\left\{
         \begin{array}{ll}
           -V_0, & \hbox{if} \ X>\sigma t \\
           0, & \hbox{if} \ X<\sigma t
         \end{array}
       \right.
\end{equation}

\begin{equation} \label{ec-0207}
\Gamma(X,t)=
\begin{cases}
0, & \hbox{if} \ X>\sigma t \\
           \Gamma_l, & \hbox{if} \ X<\sigma t,
\end{cases}
\end{equation}
where $\Gamma_l$ and $\sigma$ are determined by the Rankine-Hugoniot conditions, that
is:
\begin{equation}\label{RHnonlinear}
\begin{split}
\Gamma_l P(\Gamma_l)&=V_0^2\\
\sigma &=-\frac{V_0}{\Gamma_l}.
\end{split}
\end{equation}
\end{theorem}

We observe that the system  (\ref{RHnonlinear}) has an unique solution $(\sigma, \Gamma_l)$ provided the first equation has an unique solution
for $\Gamma_l$. We  denote such solution by $S(V_0)$ and because of the relation between $V_0$ and $\Gamma_l$ by $S(\Gamma_l)$ as well.\\
Concerning solvability of the first equation,  we notice the following fact.

\begin{lemma} \label{unicidad}
Let $P(0)=0$, $\lim\limits_{\Gamma\rightarrow -1^{+}}P(\Gamma)=-\infty $ and for all $\Gamma \in (-1,0)$,
 $P'(\Gamma)>0$. Then for each $V_0>0$ there exists an  unique $\Gamma_l \in (-1,0)$ such that $\Gamma_l P(\Gamma_l)=V_0^2$.

\end{lemma}
\begin{proof}
Let $H(\Gamma):=\Gamma P(\Gamma).$ Then $H'(\Gamma)=P(\Gamma)+\Gamma P'(\Gamma)<0$ 
for all
$\Gamma \in (-1,0)$ therefore $H(\Gamma)$
is decreasing. Since $H(0)=0$ and  $\lim\limits_{\Gamma\rightarrow -1^{+}}P(\Gamma)=-\infty, $ thus, we conclude that for each $V_0>0$, there exists an unique
$\Gamma_l \in (-1,0)$ such that $\Gamma_l P(\Gamma_l)=V_0^2$.
\end{proof}
\textbf{Proof of Theorem \ref{theoremA}}

  We verify that $S(\Gamma_l)$ is indeed a weak solution of $IBVP_{V_0}$ (\ref{ec-p}),(\ref{ec281}), i.e. we verify that $S(\Gamma_l)$ satisfies the
equations (\ref{ws1}) and (\ref{ws2}) for all test functions $\varphi$ and $\psi$, with $\varphi$ restricted by  the condition $\varphi(0,t)=0$.
We also assume that $P(0)=0.$
Here $f(X)=-V_0$, $g(X)=0$ and $h(t)=0$ in (\ref{ec3-2}).
 We verify (\ref{ws1}) first. Its left-hand sized is:
\begin{align*}
&\int_0^\infty \int_0^\infty (V\varphi_t-P(\Gamma)\varphi_X)dtdX+\int_0^\infty f(X)\varphi(X,0)dX \\
&=-V_0 \int_0^\infty \int_0^{X/\sigma} \varphi_t dt dX + \int_0^\infty \int_{X/\sigma}^\infty 0\cdot \varphi_t dtdX-P(\Gamma_l) \int_0^\infty \int_0^{\sigma t}\varphi_X dXdt\\&- P(0)\int_0^\infty \int_{\sigma t}^\infty  \varphi_X dXdt- V_0\int_0^\infty \varphi(X,0)dX\\
&=-V_0\int_0^\infty  \varphi(X,X/\sigma)dX+V_0\int_0^\infty \varphi(X,0)dX-P(\Gamma_l)\int_0^\infty \varphi(\sigma t,t)dt\\
&+P(\Gamma_l) \int_0^\infty \varphi(0,t)dt-V_0\int_0^\infty \varphi(X,0)dX\\
&=-V_0\int_0^\infty  \varphi(X,X/\sigma)dX -\frac{P(\Gamma_l)}{\sigma}\int_0^\infty  \varphi(X,X/\sigma)dX,
\end{align*}
which is zero due to   (\ref{RHnonlinear}), so that (\ref{ws1}) holds.\\
Next, for the left-hand sized of (\ref{ws2}) we have:
\begin{align*}
&\int_0^\infty\int_0^\infty (\Gamma \psi_t-V\psi_X)dtdX\\
&=\int_0^\infty\int_0^{X/\sigma} 0\cdot \psi_t dtdX+ \Gamma_l\int_0^\infty\int_{X/\sigma}^\infty  \psi_t dtdX-\int_0^\infty\int_0^{\sigma t} 0\cdot \psi_X dXdt\\
& + V_0 \int_0^\infty\int_{\sigma t}^\infty  \psi_X dXdt\\
&=\Gamma_l\int_0^\infty\int_{X/\sigma}^\infty  \psi_t dtdX +V_0 \int_0^\infty\int_{\sigma t}^\infty  \psi_X dXdt\\
&=-\Gamma_l\int_0^\infty  \psi(X,X/\sigma) dX-\frac{V_0}{\sigma}\int_0^\infty  \psi(X,X/\sigma)dX,
\end{align*}

which is zero due to (\ref{RHnonlinear}) so that (\ref{ws2}) holds.

%

%

\section{Entropy condition and entropy solution for a p-system}\label{c5}

 The 
 entropy/entropy-flux pair for a p-system is a pair of real valued $C^2(\mathbb{R}^2)$ functions
$\Phi(V,\Gamma)$ and $\Psi(V,\Gamma)$, where $\Phi$ is  convex, and  such that
\begin{equation}\label{flux}
D\Phi(V,\Gamma)DF(V,\Gamma)=D\Psi(V,\Gamma)
\end{equation}
where $F(V,\Gamma)=(-P(\Gamma),-V)$. Working out that condition one obtains

\begin{equation}\label{ec-6}
\begin{split}
\Psi_V&=-\Phi_\Gamma\\
\Psi_\Gamma &=-P'(\Gamma)\Phi_V.
\end{split}
\end{equation}
Now, the integrability condition of the system (\ref{ec-6}) for $\Psi$ is
\begin{equation}\label{integrability}
\Phi_{\Gamma\Gamma}-P'(\Gamma)\Phi_{VV}=0.
\end{equation}
Given a  convex function $\Phi$ that fulfills this equation we can obtain $\Psi$ by solving the system (\ref{ec-6}).

  \begin{definition}\label{entropia2}
A weak solution $V(X,t)$, $\Gamma(X,t)$ of an IBVP, is an entropy solution provided for each nonnegative $\varphi\in C^{\infty}_0((0,\infty)\times(0,\infty))$ and for each entropy/entropy-flux pair $\Phi,\Psi$ it holds
\begin{equation}\label{ec-7}
\int_0^\infty \int_0^\infty [\Phi(V,\Gamma)\varphi_t(X,t)+\Psi(V,\Gamma)\varphi_X(X,t)]dXdt\geq 0.
\end{equation}
We refer to (\ref{ec-7}) as the entropy condition corresponding to $(\Phi,\Psi)$.
\end{definition}
\begin{remark} \label{trivial-sol}
If a trivial solution $(V,\Gamma)=(0,0)$ is a solution of an IBVP then it is an entropy solution.
\end{remark}

The following \textit{Proposition} translates the entropy condition (\ref{ec-7}) into a jump condition for
piecewise continuous weak solutions.
  \begin{proposition}
  Suppose that $\mathbf{u}=(V,\Gamma)$ is a piecewise continuous weak solution of  (\ref{c-law2}) that satisfies the entropy condition corresponding
  to $(\Phi,\Psi)$. Suppose
   $\mathbf{u}$ has a jump discontinuity along a shock curve with slope  $\sigma$. Then

  \begin{equation}\label{jump}
   \sigma [\Phi(\mathbf{u})]-[\Psi(\mathbf{u})]\geq 0.
   \end{equation}
   We call (\ref{jump}) the entropy jump condition corresponding to $(\Phi,\Psi)$.
  \end{proposition}
\begin{proof}
  We  demonstrate (\ref{jump}), for our solution,  (\ref{ec-0206}) and (\ref{ec-0207}).

Here $\mathbf{u}=(V,\Gamma)$ satisfies the inequality (\ref{ec-7}), therefore we get



\small{$$ \Big(\int_0^{\infty} \varphi(X,X/\sigma)dX\Big)(\Phi(-V_0,0)-\Phi(0,\Gamma))+\Big(\int_0^{\infty} \varphi(t\sigma,t)dt\Big)(\Psi(0,\Gamma)-\Psi(-V_0,0))\geq 0.$$}

And after the substitution   $X=t\sigma$,  in the integral with respect to $t$, we obtain:


$$
\Big(\int_0^{\infty} \varphi(X,X/\sigma)dX\Big) \big[(\Phi(-V_0,0)-\Phi(0,\Gamma)) + \frac{1}{\sigma}((\Psi(0,\Gamma)-\Psi(-V_0,0))\big]\geq 0.
$$

Here $\varphi \geq 0$, therefore this last inequality is equivalent to

\begin{equation} \label{eec01}
\Phi(-V_0,0)-\Phi(0,\Gamma)\geq \frac{\Psi(-V_0,0)-\Psi(0,\Gamma)}{\sigma},
\end{equation}
whose compact form is  (\ref{jump}).
\end{proof}



The following Theorem states that in the case of genuine nonlinear systems, the entropy condition is satisfied for $\Gamma$
sufficiently close to zero.

\begin{theorem}
 \label{propo-import}
If $P(0)=0$, $P'(0)>0$ and $P''(0)<0$,  then for each entropy/entropy-flux pair $(\Phi,\Psi)$, where $\Phi$ is strictly convex,  $S(\Gamma)$ satisfies the entropy condition corresponding to $(\Phi,\Psi)$,  for all $\Gamma$ sufficiently close to zero and $\Gamma\leq 0$.
\end{theorem}
\begin{proof}\
The proof is based on Taylor's expansion formula. 
We notice that for $\Gamma=0$, $S(0)=0$. Therefore by \textit{Remark} \ref{trivial-sol} this solution is an entropy solution.\\
Now, we consider $\Gamma <0$.
Let $\epsilon=-V_0$, then  $\epsilon=-\sqrt{\Gamma P(\Gamma)},$ and $\sigma=\frac{\epsilon}{\Gamma} \ .$\
Define $$E(\Gamma)=\frac{\epsilon}{\Gamma}[\Phi(0,\Gamma)-\Phi(\epsilon,0)]-[\Psi(0,\Gamma)-\Psi(\epsilon,0)].$$
Consequently,  the entropy jump condition, (\ref{jump}), holds if and only if 
$E(\Gamma)\leq 0. $
We now let a ``prime'' indicate differentiation with respect to $\Gamma$ .\\
Notice that 
$$\lim_{\Gamma\rightarrow 0^{-}}\Big(\frac{\epsilon}{\Gamma}\Big)=\sqrt{P'(0)} \
\text{and} \ \lim_{\Gamma\rightarrow 0^{-}}\Big(\frac{\epsilon}{\Gamma}\Big)'=\frac{P''(0)}{\sqrt{P'(0)}}.
$$
Thus
$\displaystyle{\lim_{\Gamma\rightarrow 0^{-}}E(\Gamma)}=0 .$
Now, using (\ref{ec-6}) we obtain 
 
$$ E'(\Gamma)=\Big(\frac{\epsilon}{\Gamma}\Big)'[\Phi(0,\Gamma)-\Phi(\epsilon,0)]+ \frac{\epsilon}{\Gamma}[\Phi_{\Gamma}(0,\Gamma)-\Phi_{\epsilon}(\epsilon,0)\epsilon'] +\Phi_{\epsilon}(0,\Gamma)P'(\Gamma)-\Phi_{\Gamma}(\epsilon,0)\epsilon'.$$
Thus  
 \begin{multline*}
\displaystyle{\lim_{\Gamma\rightarrow 0^{-}}E'(\Gamma)= \sqrt{P'(0)}\Big(\Phi_{\Gamma}(0,0)-\Phi_{\epsilon}(0,0) \sqrt{P'(0)}\Big)}\\+\Phi_{\epsilon}(0,0)P'(0)-\Phi_{\Gamma}(0,0) \sqrt{P'(0)}=0.
\end{multline*}
Furthermore,
\begin{multline*}
\displaystyle{E''(\Gamma)=\Big(\frac{\epsilon}{\Gamma}\Big)'' \Big(\Phi(0,\Gamma)-\Phi(\epsilon,0)\Big)+2\Big(\frac{\epsilon}{\Gamma}\Big)'
\Big(\Phi_{\Gamma}(0,\Gamma)-\Phi_{\epsilon}(\epsilon,0)\epsilon'\Big)}\\ +\Big(\frac{\epsilon}{\Gamma}\Big)\Big(\Phi_{\Gamma\Gamma}(0,\Gamma)
-\epsilon ''\Phi_{\epsilon}(\epsilon,0)-\Phi_{\epsilon\epsilon}(\epsilon,0)(\epsilon')^2\Big)  +
P''(\Gamma)\Phi_{\epsilon}(0,\Gamma)\\+P'(\Gamma)\Phi_{\epsilon\Gamma}(0,\Gamma)-\epsilon''\Phi_{\Gamma}(\epsilon,0)-(\epsilon')^2\Phi_{\Gamma\epsilon}(\epsilon,0). \end{multline*}
Hence 
\begin{multline*}
\displaystyle{\lim_{\Gamma\rightarrow 0^{-}}E''(\Gamma)=}2\Bigg(\frac{P''(0)}{4\sqrt{P'(0)}}\Bigg)\Bigg[\Phi_{\Gamma}(0,0)
-\Phi_{\epsilon}(0,0) \sqrt{P'(0)}\Bigg]\\ + \sqrt{P'(0)}\Bigg[ \Phi_{\Gamma \Gamma}(0,0)-\frac{P''(0)}{2\sqrt{P'(0)}}\Phi_{\epsilon}(0,0)-\Phi_{\epsilon\epsilon}(0,0)
P'(0)\Bigg]\\+P''(0)\Phi_{\epsilon}(0,0)+P'(0)\Phi_{\epsilon\Gamma}(0,0)-\frac{P''(0)}{2\sqrt{P'(0)}}\Phi_{\Gamma}(0,0)
-P'(0)\Phi_{\Gamma\epsilon}(0,0)=0.
\end{multline*}
Finally,
\begin{equation*}
\begin{split}
E'''(\Gamma)=&\Big(\frac{\epsilon}{\Gamma}\Big)''' \Big(\Phi(0,\Gamma)-\Phi(\epsilon,0)\Big) +3\Big(\frac{\epsilon}{\Gamma}\Big)''\Big(\Phi_\Gamma(0,\Gamma)-\Phi_\epsilon(\epsilon,0)\epsilon'\Big)+\\
&3\Big(\frac{\epsilon}{\Gamma}\Big)' \Big(\Phi_{\Gamma\Gamma}(0,\Gamma)-\epsilon^{ \prime \prime} \Phi_\epsilon(\epsilon,0)-(\epsilon^ \prime)^2\Phi_{\epsilon\epsilon}(\epsilon,0)\Big)+\\
& \Big(\frac{\epsilon}{\Gamma}\Big) \Big[\Phi_{\Gamma\Gamma\Gamma}(0,\Gamma)-\epsilon'''\Phi_\epsilon(\epsilon,0)-3\epsilon'\epsilon''\Phi_{\epsilon\epsilon}(\epsilon,0)
-(\epsilon')^3\Phi_{\epsilon\epsilon\epsilon}(\epsilon,0) \Big]+\\
& P'''(\Gamma)\Phi_\epsilon(0,\Gamma)+2P''(\Gamma)\Phi_{\epsilon\Gamma}(0,\Gamma)+
P'(\Gamma)\Phi_{\epsilon\Gamma\Gamma}(0,\Gamma)
-\epsilon'''\Phi_\Gamma(\epsilon,0)\\
&-3\epsilon'\epsilon''\Phi_{\Gamma\epsilon}(\epsilon,0)-(\epsilon')^3\Phi_{\Gamma \epsilon\epsilon}(\epsilon,0),
\end{split}
\end{equation*}\\



therefore 
\begin{multline*}
\displaystyle{\lim_{\Gamma\rightarrow 0^{-}}E'''(\Gamma)=-3\Bigg[\frac{(1/4)(P''(0))^2-(2/3)P'(0)P'''(0)}{4\big(\sqrt{P'(0)}\big)^3}\Bigg]}
\Big[\Phi_{\Gamma}(0,0)-\Phi_{\epsilon}(0,0)\sqrt{P'(0)}\Big] \\
+\frac{3P''(0)}{4\sqrt{P'(0)}}\Bigg[\Phi_{\Gamma\Gamma}(0,0)- \frac{P''(0)}{2\sqrt{P'(0)}}\Phi_{\epsilon}(0,0)-\Phi_{\epsilon\epsilon}(0,0)P'(0)\Bigg]\\
+\sqrt{P'(0)}\Bigg[\Phi_{\Gamma\Gamma\Gamma}(0,0)-\Bigg(\frac{4P'(0)P'''(0)-
(3/2)(P''(0))^2}{8(P'(0))^{3/2}}\Bigg)\Phi_{\epsilon}(0,0)
\\-\frac{3\sqrt{P'(0)}P''(0)}{2\sqrt{P'(0)}}\Phi_{\epsilon\epsilon}(0,0)-\big(\sqrt{P'(0)}\big)^3\Phi_{\epsilon\epsilon\epsilon}(0,0)\Bigg]
+P'''(0)\Phi_{\epsilon}(0,0)\\+2P''(0)\Phi_{\epsilon\Gamma}(0,0)+P'(0)\Phi_{\epsilon \Gamma\Gamma}(0,0)-
\Bigg(\frac{4P'(0)P'''(0)-(3/2)(P''(0))^2}{8(P'(0))^{3/2}}\Bigg)\Phi_{\Gamma}(0,0)
\\ -\frac{3\sqrt{P'(0)}P''(0)}{2\sqrt{P'(0)}}\Phi_{\Gamma\epsilon}(0,0)-\big(\sqrt{P'(0)}\big)^3\Phi_{\Gamma\epsilon\epsilon}(0,0).
\end{multline*}
However from (\ref{integrability}) it follows that 

$$\Phi_{\Gamma\Gamma}(\epsilon,\Gamma)-\Phi_{\epsilon\epsilon}(\epsilon,\Gamma)P'(\Gamma)=0,$$
$$\Phi_{\Gamma\Gamma\epsilon}(\epsilon,\Gamma)-\Phi_{\epsilon\epsilon\epsilon}(\epsilon,\Gamma)P'(\Gamma)=0,$$
and
$$\Phi_{\Gamma\Gamma\Gamma}(\epsilon,\Gamma)-P''(\Gamma)\Phi_{\epsilon\epsilon}(\epsilon,\Gamma)-\Phi_{\epsilon\epsilon\Gamma}(\epsilon,\Gamma)P'(\Gamma)=0.$$
Consequently
\begin{equation} \label{derivada3}
\displaystyle{\lim_{\Gamma\rightarrow 0^{-}}E'''(\Gamma)}=-\frac{1}{2}P''(0) [\sqrt{P'(0)}\Phi_{\epsilon\epsilon}(0,0)
-\Phi_{\Gamma\epsilon}(0,0)].
\end{equation}
On the other hand, since $\Phi$ is strictly convex,  we know that for all nonzero $(a,b)\in \mathbb{R}^2$
\begin{equation}\label{ec-041}
\Phi_{\epsilon\epsilon}(0,0)a^2+2\Phi_{\epsilon\Gamma}(0,0)ab+\Phi_{\Gamma\Gamma}(0,0)b^2 > 0.
\end{equation}
Now, we demonstrate that the expression
$
\sqrt{P'(0)}\Phi_{\epsilon\epsilon}(0,0)-\Phi_{\Gamma\epsilon}(0,0)
$
in (\ref{derivada3}) can be rewritten in a form of the left hand size of (\ref{ec-041})
with $a$ and $b$ appropriately  chosen. To prove that, we modify this expression by an additive, equal to zero term  $\alpha\Phi_{\Gamma\Gamma}(0,0)-
\alpha\Phi_{\epsilon\epsilon}(0,0)P'(0)$, where $\alpha$ to be determined.\\
Thus $a$ and $b$ have to be chosen so that
\begin{equation}\label{ec-042}
 (\sqrt{P'(0)}-\alpha P'(0))\Phi_{\epsilon\epsilon}-\Phi_{\Gamma\epsilon}+\alpha\Phi_{\Gamma\Gamma}=
 \Phi_{\epsilon\epsilon}a^2+2\Phi_{\epsilon\Gamma}ab+\Phi_{\Gamma\Gamma}b^2
\end{equation}
holds, where the derivatives of $\Phi$ are at $(0,0)$.\\
Consequently we require that
\begin{eqnarray*}
a^2&=& \sqrt{P'(0)}-\alpha P'(0),\\
2ab&=&-1,\\
\alpha&=& b^2.
\end{eqnarray*}
Solving the system for $a$, $b$ and $\alpha$, we get
$a=\pm \frac{\sqrt[4]{P'(0)}}{\sqrt{2}},$
$b=\mp \frac{1}{\sqrt{2} \sqrt[4]{P'(0)}},$
and $\alpha=\frac{1}{2 \sqrt{P'(0)}}.$
In this way, we conclude from (\ref{derivada3}) that
$\displaystyle{\lim_{\Gamma\rightarrow 0^{-}}}E'''(\Gamma)>0.$
Now, to conclude the proof, we use the 
  following two \textit{lemmas}, whose proofs are straightforward.
\begin{lemma} \label{lema1.1}
Let $f$ and $f'$ be continuous functions for $\Gamma<0$. If $\lim\limits _{\Gamma \rightarrow 0^{-}}f(\Gamma)=0$ and
$\lim\limits _{\Gamma \rightarrow 0^{-}}f'(\Gamma)>0$, then there exists $\Gamma_0 \in (-1,0)$, such that  $f(\Gamma)<0$ for $\Gamma_0<\Gamma<0$.

\end{lemma}




\begin{lemma} \label{lema1.2}
Let $f$ and $f'$ be continuous functions for $\Gamma<0$. If $\lim\limits _{\Gamma \rightarrow 0^{-}}f(\Gamma)=0$ and
$\lim\limits _{\Gamma \rightarrow 0^{-}}f'(\Gamma)<0$ then,  $f(\Gamma)>0$ for $\Gamma_1<\Gamma<0$, where
$\Gamma_1 \in (-1,0)$.
\end{lemma}
Applying  \textit{Lemma} \ref{lema1.1} with $f(\Gamma)=E''(\Gamma)$, we conclude that $E''(\Gamma)<0$  for $\Gamma$ close to zero and,
using \textit{Lemma} \ref{lema1.2}, with $f(\Gamma)=E'(\Gamma)$, we infer that $E'(\Gamma)>0$ near zero. Thus, using \textit{lemma} \ref{lema1.1} again, now, with  $f(\Gamma)=E(\Gamma)$, we conclude that $E(\Gamma)<0$ near zero.
\end{proof}

%
%

\subsection{Entropy condition for a solution of $IBVP_{V_{0}}$}
It is difficult to describe explicitly all entropy functions $\Phi$. Nevertheless, employing separation of variables, we can figure out
one of them, which we shall call a standard entropy function.\\
  For this purpose we set  $\Phi(V,\Gamma)=a(V)+b(\Gamma)$, where $a$ and $b$ are functions to be determined. Substituting it in (\ref{integrability}), we obtain
$a''(V)=\frac{b''(\Gamma)}{P'(\Gamma)}.$
Since $V$ and $\Gamma$ are independent variables, therefore there exists a constant denoted by $c$ such that
$a''(V)=\frac{b''(\Gamma)}{P'(\Gamma)}=c,$
which implies
$a(V)=c\frac{V^2}{2}+c_1V+c_2; \ c_1,c_2\in \mathbb{R},$
and
$b(\Gamma)=c\int_0^\Gamma P(w) dw+c_3\Gamma +c_4; \ c_3,c_4 \in \mathbb{R}.$
Thus, putting $c_1=c_2=c_3=c_4=0$, we get
$\Phi(V,\Gamma)=c\frac{V^2}{2}+ c\int_0^\Gamma P(w)dw.$
Substituting $\Phi$ into (\ref{ec-6}), results in
\begin{equation}\label{ec-6-01}
\begin{split}
\Psi_V&=-c P(\Gamma),\\
\Psi_\Gamma &=-cV P'(\Gamma).
\end{split}
\end{equation}
A  solution of the system (\ref{ec-6-01}) is
$\Psi(V,\Gamma)=-cVP(\Gamma).$
Notice that, strict convexity of $\Phi$ implies that 
$c >0$ and $P'(\Gamma)>0$. Thus, without loss of generality, we may  put  $c=1$, thereby 
we obtain 
\begin{equation}\label{eec03}
\Phi(V,\Gamma)= \frac{V^2}{2}+\int_0^\Gamma P(w)dw
\end{equation}
and
\begin{equation}\label{eec04}
\Psi(V,\Gamma)=-P(\Gamma)V.
\end{equation}
The function (\ref{eec03}) is well known entropy function for a p-system, \cite{EVA},
which we call a standard entropy function.\\
For the solution $S(\Gamma_l)$, (\ref{RHnonlinear}),   the condition (\ref{ec-7}) can be simplified into (\ref{jump}). Here,  $P(0)=0$, $ \Phi(-V_0,0)=\frac{V_0^2}{2}, \ \ \Phi(\Gamma,0)=\int_0^{\Gamma} P(w)dw,\ \Psi(-V_0,0)=0,$ and
$  \Psi(0,\Gamma)=0$, so that (\ref{jump})   becomes
\begin{equation}\label{entropy-particular}
2\int^{\Gamma_l}_0 P(w)dw \leq \Gamma_l P(\Gamma_l).
\end{equation}
This is    the entropy condition for $S(\Gamma_l)$ corresponding to  a standard entropy function, (\ref{eec03}), and $\Psi$ given by (\ref{eec04}).
\begin{remark}
The assertion of Theorem \ref{propo-import} does not say how far from $0$ the inequality still
holds or it already does not hold. It is rather difficult, except of a linear case, to answer this question
without having more particular information about the entropy functions. That is why we concentrate ourselves
on studying the inequality (\ref{entropy-particular}),
for previously listed models of elastic materials.
\end{remark}
We notice the following facts, which, clarify an importance of genuine nonlinearity condition in studying  the entropy condition (\ref{entropy-particular}).

\begin{lemma}\label{propo7}
If $P(0)=0$ and $P^{''}(\Gamma)<0$ for all $\Gamma \in (-1,0)$, then $S(\Gamma_l)$ satisfies (\ref{entropy-particular}) for all
$\Gamma_l \in (-1,0]$.
\end{lemma}
\begin{proof}
 Consider the function
  $G(\Gamma_l)=2\int_0^{\Gamma_l}P(w)dw-\Gamma_l P(\Gamma_l) $;
   and notice that $G(\Gamma_l)$ is increasing for all
   $\Gamma_l\in(-1,0)$.
\end{proof}
Similarly we have the following proposition: 
\begin{lemma}\label{propo8}
If $P(0)=0$ and $P^{''}(\Gamma_l)>0$ for $\overline{\Gamma}<\Gamma_l<0$, where $\overline{\Gamma} \in (-1,0)$, then
$S(\Gamma_l)$ does not  satisfy (\ref{entropy-particular}).  Therefore $S(\Gamma_l)$ does not satisfy the  entropy condition for $\overline{\Gamma}<\Gamma_l<0$.
\end{lemma}

\section{Results on hyperbolicity, genuine nonlinearity and entropy condition with a standard entropy function}

In this section we present the results about the conditions of hyperbolicity ($P'(\Gamma)>0$),
genuine nonlinearity ($P''(\Gamma)<0$) and the entropy condition 
(see equation ($\ref{entropy-particular}$))
for the models under consideration. To attain this goal, we use 
basic techniques of differential calculus and the Maple software to perform 
symbolic computation and to study the graphs of functions.

In some cases it is convenient to use instead of $\Gamma$
a variable $s=\Gamma+1$, restricted by $s>0$, since $\Gamma$ is subject to
$\Gamma>-1$. 

\begin{enumerate}

\item St.Venant-Kirchhoff
\begin{enumerate}
\item It is hyperbolic for all $\Gamma>-1+1/\sqrt{3}$.
\item The condition of genuine nonlinearity is satisfied for 
all $\Gamma>-1$.
\item For all $s\in (0,1)$, $S(\Gamma_l)$ do not satisfy the entropy condition.

\end{enumerate}

\item Kirchhoff modified

\begin{enumerate}

\item It is hyperbolic for all $\Gamma>-1$ provided a parameter
$\frac{\lambda}{\mu}$ satisfies $\alpha_1<\alpha<\alpha_2$
where $\alpha_1$ and $\alpha_2$ are two positive solutions 
of the following equation 
$$6(5+4\log 6)\alpha=1+12\alpha\log(3+3\sqrt{1+12\alpha})+\sqrt{1+12\alpha}.$$
An approximate inequality for $\alpha$ is 
$$0.0446567295<\alpha<1732.05696$$

\item 
$P''(\Gamma)<0$ holds only for all $s\in (0,S_\alpha)$, where
$$S_\alpha=\bigg[\frac{\alpha}{2}\operatorname{Lambert}W\bigg(\frac{12e^{6}}{\alpha}\bigg)\bigg]^{1/4}$$
and where $\operatorname{Lambert}W$ is the inverse of the function 
$we^{w}$. Consequently $P''(s)<0$ holds for all $s\in (0,1]$
iff $s_\alpha>1$, what is equivalent to $\alpha>2$.

\item 
\begin{itemize}
\item If $\alpha\geq 2$, then $S(\Gamma_l)$ satisfies the entropy condition
for all $s\in (0,1)$. 
\item If $0<\alpha <2$, then $S(\Gamma_l)$ satisfies the entropy condition
for all $s\in (0,s_e]$ and does not for $s\in (s_e,1)$, where $s_e$
is a unique solution in $(0,1)$ of the equation 
$$\frac{s(s+1)(1-s)^{3}}{2(s-1-s\ln s)\ln s}=\alpha.$$
\end{itemize}
\end{enumerate}

\item Ogden
\begin{enumerate}

\item It is hyperbolic for all $\Gamma>-1$.
\item Satisfies that $P''(s)<0$ for all $s>0$. 
\item $S(\Gamma_l)$ satisfies the entropy condition, for all $s\in (0,1]$. 
\end{enumerate}

\item Blatz-Ko-Ogden
\begin{enumerate}

\item There are two parameters involved $\beta=\frac{\lambda}{2\mu}$ and
$f\in (0,1)$.
\begin{itemize}
\item If $\beta\geq 1/2$, it is hyperbolic for all $\Gamma>-1$.
\item If $0<\beta<1/2$, the hyperbolicity condition requires a
restriction for $f$ of the form $f>f_\beta$, where $f_\beta$
is a certain number in $(0,1)$ determined according to 
$f_\beta=\max_{s>s_\beta}Q(s,\beta)$
where 
$$Q(s,\beta)=\frac{s^{2\beta}[(1-2\beta)s^{2\beta+2}-3]}{s^2[(1+2\beta)+s^{2\beta+2}]+s^{2\beta}[(1-2\beta)s^{2\beta+2}-3]}$$
and 
$$s_\beta=\big(\frac{3}{1-2\beta}\big)^{\frac{1}{2\beta+2}},$$
here $f_\beta\leq \frac{1-2\beta}{1-2\beta+s_\beta^{2-2\beta}}.$
\end{itemize}

\item 
\begin{itemize}
\item If $\beta\in [1/2,1]$ then 
$P''(s)<0$ for all $s>0$. 
\item If $\beta \in (0,1/2)\cup(1,\infty)$, then for all $s\leq s_0$,
where $$s_0=\bigg[\frac{6}{(2\beta-1)(\beta-1)}\bigg]^{\frac{1}{2\beta+2}}$$
it holds $P''(s)<0$.
\item If $\beta \in (0,1/2)\cup(1,\infty)$ and $s>s_0$, then 
there exists $s_2$ such that $P''(s)<0$ up to $s_2$ and
then it changes its sign. 

\end{itemize}
\

\item 
\begin{itemize}
\item If $0<\beta\leq 5/2$ then $P''(s)<0$ for all $s\in (0,1]$. 
\item An experimentation with plots indicates that for a given value of
$\beta>5/2$ there exists $f_\beta\in (0,1)$ such that, 
$S(\Gamma_l)$ satisfies the entropy condition, for all $s\in (0,1]$,
provided $f\geq f_\beta$. If however $f<f_\beta$, then
there exists $s_\beta\in (0,1]$ such that the condition
holds for all $s\in (0,s_\beta)$ and does not
for $s\in (s_\beta,1)$, while at $s=1$ it holds again.
We have been able to confirm theoretically such behavior of the condition
only for $\beta=n/2$ , where $n$ is an integer and $n>5$.  

\end{itemize}

\end{enumerate}

\end{enumerate}







\newpage

\section{Conclusions}

\begin{enumerate}
\item A definition of a weak solution of an initial and boundary problem for
a p-system, in the first qudrant of the $Xt$-plane, is provided. There are
two unknown functions $V\left( X,t\right) $ and $\Gamma \left( X,t\right) $.
Consequently there are two initial conditions (at $t=0$) and only one
boundary condition (at $X=0$). There are four types of boundary conditions
considered: the first, (\ref{ec3-2}), for $V\left( 0,t\right) $, the second,
(\ref{ec8}), for $\Gamma \left( 0,t\right) $ and the other two are mixed boundary
conditions involving $V\left( 0,t\right) $ and $\Gamma \left( 0,t\right) $,
(\ref{ec-m1}) and  (\ref{ec-m2}) respectively. The first two types of boundary conditions are
particular cases of the other two. All of that is consistent with what is
known in case of classical solutions of linear systems, \cite{EVA}.

\item A particular weak solution of a p-system, called a compression shock
is constructed. It satisfies the initial and boundary conditions given by (%
\ref{ec281}), which is a particular case of (\ref{ec3-2}). This solution,
denoted by $S\left( \Gamma _{l}\right) $, which can be interpreted as an impact velocity. 
$S\left( \Gamma _{l}\right) $
is constant by parts, having jump discontinuities of $V$ and $\Gamma $ along
the line $X=\sigma t$; $\left( V,\Gamma \right) =\left( -V_{0},0\right) $,
for $X>\sigma t$ $\ $and $\left( V,\Gamma \right) =\left( 0,\Gamma
_{l}\right) $, $X<\sigma t$, where the constants $\sigma >0$ and $\Gamma
_{l}<0$ are solutions of the Rankine-Hugoniot conditions. 


\ \


 \item  For the St.Venant-Kirchhoff model $S(\Gamma _{l})$ does not satisfy the entropy condition. Consequently we can consider this model as inadequate to
describe the compression shock.
For the Kirchhoff modified, Ogden and Blatz-Ogden models we can verify that they satisfy, under certain restrictions on the parameters, the hypothesis of the  Theorem \ref{propo-import}.  Therefore for those models $S(\Gamma _{l})$  satisfies the entropy condition, for $%
\Gamma _{l}$ sufficiently close to zero.

\item  The Theorem \ref{propo-import}   does not provide an exact information about the interval for  $\Gamma_l$ in which
the entropy condition holds. That is why we concentrate on the entropy
condition with a well known in literature \cite{EVA}, entropy/entropy-flux pair $(\Phi ,\Psi )$, which
we call a standard entropy/entropy-flux pair.
 We  provide  the conditions for the parameters $\mu $,$\lambda $, $f$ and for $\Gamma_l$, under which $S(\Gamma _{l})$ fulfills
 the entropy condition with this standard entropy function. This discussion is  complete, except of the Blazt-Ko and Ogden model
for  $\beta>\frac{5}{2}$. In this case we  clarify the validity of the entropy condition only for $\beta=\frac{n}{2}$, where $n$ is an integer number greater than $5$.

 \item An open question remains about the entropy condition with a general entropy function.
\end{enumerate}

\newpage
\appendix

\begin{appendices}

\section{Numerical comparison of the compression shocks for various models}

In this section we obtain numerical values of $\Gamma_l$, for the compression
shock corresponding to given values of $V_0$; more specifically we use 
$\widetilde{V_0}=\frac{\rho_0V_0^2}{\mu}.$
We do this for the following models: Modified Krchhoff, Ogden and 
Blatz-Ko-Ogden.

Here $\Gamma_l$ is determined by the first equation in $(\ref{RHnonlinear})$, which after substituting
$\lambda=2\mu \beta$ can be rewritten in the form $Q(\Gamma)=\widetilde{V_0}$,
where  

$$Q(\Gamma)=\Gamma\bigg(\mu(1+\Gamma)^3-(1+\Gamma)+2\beta \frac{\ln(1+\Gamma)}{(1+\Gamma)} \bigg) \ \ \text{(Modified Kirchhoff)}.$$
$$Q(\Gamma)=\Gamma\bigg(2\beta \Gamma+\frac{(2+\Gamma)\Gamma}{\Gamma+1}\bigg)
\ \ \text{(Ogden model)}.$$
$Q(\Gamma)=\Gamma(1+\Gamma)\Bigg\{f \bigg[1-(1+\Gamma)^{-2\beta-2}\bigg]+\frac{(1-f)}{(1+\Gamma)^4}\bigg[(1+\Gamma)^{2\beta+2}-1\bigg]\Bigg\}\\
\ \ \text{(Blatz-Ko-Ogden model)}.$









\begin{table}[!htb]
\centering
\begin{tabular}{|c|c|c|c|c|}
 \hline 
   \diagbox{$\widetilde{V_0}$}{$\Gamma$}& Ogden & M.Kirchhoff& Blatz-Ko $(f = 0.25)$ & Blat-Ko $(f = 0.5)$ \\ 
 \hline 
 \hline
0.1	& -0.1912	&-0.217420 &	  -0.373581	& -0.447296\\
0.25	& -0.2929	&-0.351888 &  -0.386761	&-0.457802\\
0.5	& -0.3978	&-0.486632 &   -0.407276	&-0.474264\\
2	& -0.6667	&-0.733399 &   -0.495098	&-0.547908\\
4	& -0.7938	&-0.818724	&  -0.559164	&-0.604841\\
10	& -0.9063	&-0.897073	&  -0.646584	&-0.684415\\
40	& -0.9753	&-0.960995	&  -0.760038	&-0.787456\\

 \hline 
 \end{tabular}  
 \caption{$\beta=0.25$}
\label{tabla1}
 \end{table}

 \begin{table}[!htb]
\centering
\begin{tabular}{|c|c|c|c|c|}
 \hline 
  \diagbox{$\widetilde{V_0}$}{$\Gamma$}& Ogden & M.Kirchhoff& Blatz-Ko $(f = 0.25)$ & Blat-Ko $(f = 0.5)$ \\ 
 \hline 
 \hline
0.1   &   -0.1764   &     -0.187793  &           -0.396762   &          -0.467705\\
 0.25  &      -0.2722&     -0.294512  &           -0.406156  &          -0.475495\\
 0.5   &      -0.3729 &    -0.401976  &            -0.42134  &          -0.488051\\
2      &     -0.6446  &   -0.638674   &          -0.495152   &          -0.550033\\
4      &     -0.7808  &   -0.739561   &          -0.556282   &          -0.603611\\
10     &     -0.9027  &   -0.842347   &          -0.643929   &          -0.682669\\
40     &     -0.9678  &  -0.9357075   &         -0.759011    &          -0.786753\\

 \hline 
 \end{tabular}  
 \caption{$\beta=0.5$}
\label{tabla2}
 \end{table}

 \begin{table}[!htb]
\centering
\begin{tabular}{|c|c|c|c|c|}
 \hline 
  \diagbox{$\widetilde{V_0}$}{$\Gamma$}& Ogden & M.Kirchhoff& Blatz-Ko $(f = 0.25)$ & Blat-Ko $(f = 0.5)$ \\ 
 \hline 
 \hline
 
   0.1 &        -0.1276  &  -0.123866 &            -0.516574 &               -0.571569\\
  0.25 &         -0.2    &  -0.189781 &            -0.518832 &               -0.573691\\
  0.5  &       -0.2798   &  -0.257764 &            -0.522632 &               -0.577243\\
  2    &       -0.5298   &  -0.440641 &            -0.546019 &               -0.598641\\
  4    &       -0.6951   &  -0.546173 &            -0.576544 &               -0.625725\\
  10   &       -0.8757   &  -0.681674 &            -0.645033 &               -0.685736\\
  40   &       -0.9731   &  -0.843239 &            -0.758247 &               -0.786419 \\

 \hline 
 \end{tabular}  
 \caption{$\beta=2$}
\label{tabla3}
 \end{table}
 
 \begin{table}[!htb]
\centering
\begin{tabular}{|c|c|c|c|c|}
 \hline 
   \diagbox{$\widetilde{V_0}$}{$\Gamma$}& Ogden & M.Kirchhoff& Blatz-Ko $(f = 0.25)$ & Blat-Ko $(f = 0.5)$ \\ 
 \hline 
 \hline
 
 0.1&         -0.0909 &    -0.145165   &          -0.646008   &               -0.68264\\
0.25 &      -0.1433   &    -0.223122   &          -0.646471   &               -0.683109\\
 0.5  &       -0.2020 &    -0.302771   &          -0.647248   &               -0.683894\\
2     &      -0.3975  &    -0.506407   &          -0.652028   &               -0.688701\\
4     &      -0.5501  &    -0.614488   &          -0.658692   &               -0.695328\\
10    &      -0.7941  &    -0.743021   &          -0.680005   &               -0.715904\\
40    &      -0.9684  &    -0.881932   &          -0.759992   &               -0.788159 \\

 \hline 
 \end{tabular}  
 \caption{$\beta=5$}
\label{tabla4}
 \end{table}

\end{appendices}

\newpage

{\footnotesize

\end{document}